\documentclass[12pt,a4paper]{article}

\usepackage[T1]{fontenc}
\usepackage[english]{babel}
\usepackage{amsmath}
\usepackage{amssymb}
\usepackage{amsthm}
\usepackage{hyperref}
\usepackage{microtype}
\usepackage{subcaption}

\usepackage{tikz}
\usetikzlibrary{calc}

\newtheorem{theorem}{Theorem}

\makeatletter
\newcommand{\STS@text}[1]{{\rm STS}$(#1)$}
\newcommand{\STS@unstarred}[1]{\ifmmode{\text{\mbox{\STS@text{#1}}}}\else{\mbox{\STS@text{#1}}}\fi}
\newcommand{\STS@starred}[1]{\ifmmode{\text{\mbox{sub-\STS@text{#1}}}}\else{\mbox{sub-\STS@text{#1}}}\fi}
\newcommand{\STSS@unstarred}[1]{\ifmmode{\text{\mbox{\STS@text{#1}s}}}\else{\mbox{\STS@text{#1}s}}\fi}
\newcommand{\STSS@starred}[1]{\ifmmode{\text{\mbox{sub-\STS@text{#1}s}}}\else{\mbox{sub-\STS@text{#1}s}}\fi}
\DeclareRobustCommand{\STS}{\@ifstar{\STS@starred}{\STS@unstarred}}
\DeclareRobustCommand{\STSS}{\@ifstar{\STSS@starred}{\STSS@unstarred}}
\makeatother

\newcommand{\NDG}{\ensuremath{\mathrm{N}_{\mathcal{D}}(G)}}
\newcommand{\NFGW}{\ensuremath{\mathrm{N}_{\mathcal{F}}(G,\mathcal{W})}}
\newcommand{\tDG}{\ensuremath{t_{\mathcal{D}}(G)}}
\newcommand{\tFGW}{\ensuremath{t_{\mathcal{F}}(G,\mathcal{W})}}

\begin{document}

\title{Enumerating Steiner Triple Systems}

\author{
Daniel Heinlein\thanks{Supported by the Academy of Finland, Grant 331044.}\phantom{ }
and Patric R. J. \"Osterg\aa rd\\
Department of Information and Communications Engineering\\
Aalto University School of Electrical Engineering\\
P.O.\ Box 15400, 00076 Aalto, Finland\\
\tt \{daniel.heinlein,patric.ostergard\}@aalto.fi
}

\date{}

\maketitle

\begin{abstract}
  Steiner triple systems (STSs) have been classified up to order 19. Earlier estimations
  of the number of isomorphism classes of STSs of order 21, the smallest open case,
  are discouraging as for classification, so it is natural to focus on the easier problem
  of merely counting the isomorphism classes. Computational approaches
  for counting STSs are here considered and lead to an algorithm that is used to
  obtain the number of isomorphism classes for order 21: $14{,}796{,}207{,}517{,}873{,}771$.
\end{abstract}



\section{Introduction}

A \emph{Steiner triple system} (STS) is a pair $(V,\mathcal{B})$,
where $V$ is a set of \emph{points} and $\mathcal{B}$ is a
set of \mbox{3-subsets} of points, called \emph{blocks}, such that
every \mbox{2-subset} of points occurs in exactly one block. The size of
the point set is the \emph{order} of the STS, and an
STS of order $v$ is denoted by \STS{v}. An \STS{v} exists if and only if
\begin{equation}\label{eq:sts}
v \equiv 1\text{ or }3\!\!\!\pmod{6}.
\end{equation}
For more information about Steiner triple systems, see~\cite{C,CR}.

An \STS{v} is \emph{isomorphic} to another \STS{v} if
there exists a bijection between the point sets that maps blocks
onto blocks; such a bijection is an \emph{isomorphism}.
An isomorphism of a Steiner triple system onto itself
is an \emph{automorphism} of the STS.
The automorphisms of an STS form a group under composition, the
\emph{automorphism group} of the Steiner triple system.

Steiner triple systems have been classified up to order 19. The numbers
of isomorphism classes are 1, 1, 1, 2, 80, and 11,084,874,829
for the admissible orders 3, 7, 9, 13, 15, and
19, respectively; see~\cite{KO04,KO06} for details about
classification of Steiner triple systems and~\cite{G} for historical
remarks and speculations about future results.
Indeed, for all these orders, representatives of
the isomorphism classes have been determined. For the smallest open
case of order 21, however, such an approach does not seem feasible
at the moment as the number of isomorphism classes has been estimated~\cite{HO}
to be greater than $10^{16}$. There are numerous studies on classification
of subclasses of Steiner triple systems of order 21
\cite{CCIL,GSK,KT,K,K05,KOTZ,KO20,KO21,LM,MPR,MR,T87,T04}.

If an instance of classifying combinatorial structures with given
parameters is infeasible with the available resources,
it might still be possible to \emph{count} the isomorphism classes.
There are several examples of such studies in the past, for
instance, for Latin squares~\cite{HKO,MMM}, one-factorizations of
complete graphs~\cite{KO09}, and Steiner triple systems with subsystems~\cite{KOP}.
A classification of \STSS{21} with nontrivial automorphisms
is available~\cite{K}, so the number of isomorphism classes can be obtained with
the Orbit--Stabilizer theorem and a count of the number of labeled
\STSS{21}.

A general approach for counting labeled combinatorial structures as fast
as possible is to divide the structures into parts and then
do the counting for each part. Such divisions come naturally for certain
objects, such as Latin squares (via Latin rectangles) and one-factorizations of
complete graphs (via sets of one-factors), but for Steiner triple systems the
situation is not so obvious. In the approach developed in the current work,
a division of the blocks into two parts is obtained via graphs that are close
to being regular. Classification of such graphs is therefore needed as an
ingredient. The main result of the current work is as follows.

\begin{theorem}\label{thm:main}
The number of isomorphism classes of Steiner triple systems of order\/ $21$
is\/ $14{,}796{,}207{,}517{,}873{,}771$.
\end{theorem}

Note that although
this work is about counting and not about classification, representatives
of each isomorphism class are \emph{seen piecewise} many times.
It might even be possible to modify
the counting algorithm to investigate also for \STSS{21} some of the properties
that were studied for \STSS{19} in~\cite{Cetal}.

The paper is organized as follows. The general approach is presented in
Section~\ref{sect:general}. In Section~\ref{sect:compute}, computational
subproblems are considered: the choice of a partition of the
blocks (Section~\ref{sect:part}), classification of graphs with
given degree sequences (Section~\ref{sect:degree}),
counting labeled systems (Section~\ref{sect:label}),
and counting and validating the number of isomorphism classes
(Section~\ref{sect:iso}).
Finally, in Section~\ref{sect:results}, the computational results are
summarized.

\section{General Approach}\label{sect:general}

The general framework considered here builds on that of~\cite{HO,KOP},
where Steiner triple systems with subsystems are considered.
Actually, the approach for classifying \STSS{19} used in~\cite{KO04}
can also be put into this framework.

In~\cite{HO,KO04,KOP}, the blocks are partitioned into three sets:
a \emph{defining set} $\mathcal{B}'$ that forms a subsystem
and sets $\mathcal{F}$ and $\mathcal{D}$
whose blocks intersect the point set of $\mathcal{B}'$ in one point
and zero points, respectively. As $\mathcal{B}'$ contains the blocks
of a subsystem---which
in~\cite{KO04} is just a sub-\STS{3}, a single block---no
block can intersect the point set of $\mathcal{B}'$ in exactly two points.

Before giving the theorem~\cite{HO,KOP} showing the full picture we need to
define configurations.
A $(v_r,b_k)$ \emph{configuration} is an incidence structure with $v$ points
and $b$ blocks, such that each block contains $k$ points, each point occurs
in $r$ blocks, and two different blocks intersect in at most one point.

\begin{theorem}\label{thm:splitting}
Let $(V,\mathcal{B})$ be an\/ \STS{v} that has a \STS*{w} $(W,\mathcal{B}')$.
Then
\begin{enumerate}
\item $\mathcal{B} = \mathcal{B}' \cup \mathcal{F} \cup \mathcal{D}$ where
$\mathcal{F}$ and $\mathcal{D}$ are the sets of blocks that intersect $W$ in\/ $1$ and\/ $0$ points, respectively,
\item $\mathcal{F} = \bigcup_{p \in W} \mathcal{F}_p$ where $\mathcal{F}_p$ is the set of blocks in $\mathcal{F}$ that contain $p$,
\item $\mathcal{F}'_p = \{B \setminus \{p\} : B \in \mathcal{F}_p\}$
with $p \in W$ is a\/ \mbox{$1$-factor} of a graph $G$ with vertices $V \setminus W$ and edges\/ $\bigcup_{p \in W} \mathcal{F}'_p$,
\item $\{\mathcal{F}'_p : p \in W\}$ is a\/ \mbox{$1$-factorization} of $G$,
\item $G$ is $w$-regular and its complement $\overline{G}$ is $(v-2w-1)$-regular, and\label{thm:splittingG}
\item $\overline{G}$ can be decomposed into a set of edge-disjoint\/ \mbox{$3$-cycles}---$\mathcal{D}$ being one possible set---which
forms a \[({(v-w)}_{(v-2w-1)/2},{((v-w)(v-2w-1)/6)}_{3})\] configuration.
\end{enumerate}
\end{theorem}

The fact that $G$ is a regular graph in
Theorem~\ref{thm:splitting}\ref{thm:splittingG} is essential.
Steiner triple systems with subsystems can therefore be constructed via a classification
of certain regular graphs. We shall now modify this approach for situations where we still
have a partition of the point set into $W$ and $V \setminus W$, but
$W$ does not induce a subsystem.
Graphs will play a central role also in the modified approach, but the graphs considered
will have also other degree sequences than those of regular graphs. Informally, one
could call the graphs nearly regular. The order of the vertices will not matter, so
we may use an abbreviated form $d_1^{n_1}d_2^{n_2}\cdots d_k^{n_k}$ for a degree sequence with
$n_i$ copies of $d_i$, $1 \leq i \leq k$.

A PBD($w$,$K$) \emph{pairwise balanced design} is a pair $(W,\mathcal{B}')$, where
$W$ is a set of $w$ points, $\mathcal{B}'$ is a set of blocks with sizes from $K$,
and every pair of distinct points occurs in exactly one block.
For a $(V,\mathcal{B})$ STS and any point set $W \subseteq V$,
the pair $(W,\mathcal{B}')$ where
\[
\mathcal{B}' = \{B \cap W : B \in \mathcal{B},\ |B \cap W| \geq 2\}
\]
is a PBD($|W|$,$\{2,3\}$).

We shall now see what happens when $W$ induces a
PBD($|W|$,$\{2,3\})$.
The
items of Theorem~\ref{thm:splittingnew} follow those
of Theorem~\ref{thm:splitting}. As the results follow directly
from definitions, the theorem is stated without proof.

\begin{theorem}\label{thm:splittingnew}
Let $(V,\mathcal{B})$ be an\/ \STS{v} and let $W \subseteq V$.
Then
\begin{enumerate}
\item $\mathcal{B} = \mathcal{B}'' \cup \mathcal{F} \cup \mathcal{D}$ where
  $ \mathcal{B}''$, $\mathcal{F}$, and $\mathcal{D}$ are the sets of blocks that intersect
  $W$ in at least $2$, exactly\/ $1$, and exactly\/ $0$ points, respectively,
\item $\mathcal{F} = \bigcup_{p \in W} \mathcal{F}_p$ where $\mathcal{F}_p$ is the set of blocks
  in $\mathcal{F}$ that contain $p$,
\item\label{thm:factor}
  $\mathcal{F}'_p = \{B \setminus \{p\} : B \in \mathcal{F}_p\}$ with $p \in W$ forms a partition of\/
  $V \setminus (W \cup W_p)$, where $W_p = \bigcup_{B: p \in B \in \mathcal{B}''} B \cap (V \setminus W)$;
  $\mathcal{F}'_p$ can also be viewed as a color class of a
  proper edge-coloring of a graph $G$ with vertices $V \setminus W$ and edges
  $\bigcup_{p \in W} \mathcal{F}'_p$,
\item $\{\mathcal{F}'_p : p \in W\}$ is a proper edge-coloring of $G$,
\item the degree of a vertex $v$ in $G$ is\/ $|W|-|\{p : v \in W_p\}|$, and
\item $\overline{G}$ can be decomposed into a set of edge-disjoint\/ \mbox{$3$-cycles}---$\mathcal{D}$
  being one possible set.
\end{enumerate}
\end{theorem}

For counting the $\STSS{v}$
via the graphs $G$ in Theorem~\ref{thm:splittingnew},
one first needs to define the PBD$(w,\{2,3\})$ induced by $W$. To get an easy
formula for the final count, the number of occurrences of this PBD should only
depend on the order $v$.
The next step is then to determine possible degrees of the graphs $G$ and classify them.
For each classified graph $G$, there is the computational task of finding
the number of proper edge-colorings of $G$
(with some additional requirements) and the number of decompositions of $\overline{G}$
into triangles. From the data of these computations, and a classification of the
\STSS{v} with nontrivial automorphisms, the number of isomorphism classes can
finally be obtained using the Orbit--Stabilizer theorem.

\section{Computational Subtasks}\label{sect:compute}

We shall now look at the subtasks of the outlined algorithm,
with details for $\STSS{21}$ at the end of each subsection and
separated from the general discussion.

\subsection{Choice of Defining Set}\label{sect:part}

The choice of the defining set, that is, the PBD$(w,\{2,3\})$ induced by $W$,
is critical for the length of the computation. The final computation is expected to be very
extensive, and a proper and justified choice cannot be made without
estimations and experiments. As for time usage, the core subproblem is that of
obtaining the number of possible sets $\mathcal{D}$ and $\mathcal{F}$ for a given
graph $G$ (Theorem~\ref{thm:splittingnew}\ref{thm:factor}). The latter number
further depends on the sets $W_p$, also defined in
Theorem~\ref{thm:splittingnew}\ref{thm:factor}, and these numbers are denoted
by \NDG{} and \NFGW{}, where $\mathcal{W}$ is
a \emph{multiset} whose elements are the sets $W_p, p \in W$.
It is important to notice that
\NDG{} is defined to be the number of decompositions
of the \emph{complement} of $G$. Algorithms for
obtaining these numbers will be discussed in Section~\ref{sect:decomp}.

The problem of classifying the graphs $G$ is less time-consuming and
even more so in an experimental phase, where randomly
generated graphs are considered.
Random graphs can be obtained in a rather
straightforward manner using a Markov chain Monte Carlo (MCMC) algorithm.
For a given degree sequence, one graph can be constructed with the Havel--Hakimi
algorithm~\cite{H1,H2}. Thereafter, an MCMC algorithm~\cite{RJB,VL}
can be applied to get a sequence of more graphs. In the core of this
algorithm is a switch where two random edges $\{a,b\}$ and $\{c,d\}$
are replaced with the edges $\{a,c\}$ and $\{b,d\}$ if the latter edges
do not already exist. As the produced graphs are only used to get rough
estimates for the main algorithm, it is not necessary to be meticulous
in tuning the details of this approach.

To optimize the computing time of the main counting algorithm, experiments
can be carried out to find a good choice of PBD$(w,\{2,3\})$s $(W,\mathcal{B}')$.
We denote $\mathcal{B}' = \mathcal{B}'_3 \cup \mathcal{B}'_2$ where the subindex
gives the block size. It is further desired that the STSs do not contain many PBDs of
the given type. Namely, the more copies there are, the more times each
STS will be encountered in the final computation and the longer the
execution time will be. In this sense, it is generally good to have $|\mathcal{B}'_3|$
large. (The extremal case of $|\mathcal{B}'_3|=0$ is referred to as an
\emph{independent set} in a Steiner triple system.)
We let $W = \{0,1,2,\ldots ,w-1\}$ in the sequel.

A central feature of the main algorithm is that it counts
$\NDG{}\NFGW{}$ labeled STSs with a
computing time in the order
of $\NDG{}+ \NFGW{}$.
It is therefore desired to have
\NDG{} and
\NFGW{}---alternatively, the computing times for these---in
approximately the same order of magnitude.

\paragraph{STS(21)}
For \STSS{21} and
an \STS{3}---a single block---as a defining set (so $w=3$), there is a huge
imbalance and determining \NDG{} is very time-consuming~\cite{KO04}. For large $w$,
there is an imbalance in the other direction:
for \STSS{21} and an \STS{7} as a defining set (so $w=7$),
$\NDG{} =0$ for most of the graphs~\cite{HO}. Therefore,
possible values of $w$ are here narrowed down to $4 \leq w \leq 6$.

An STS(21) contains 1260, 4725, and 10584 PBDs that are independent sets
for $w = 3,4$, and 5, respectively~\cite{FGG}. These numbers
are up to more than an order of magnitude bigger than for other choices
as we shall soon see, and they support the choice of maximizing $|\mathcal{B}'_3|$.

For \STSS{21} with PBD$(w,\{2,3\})$s $(W,\mathcal{B}')$,
$4 \leq w \leq 6$ restricted in the aforementioned way,
experimental results are presented in Table~\ref{tbl:est}.
Only the blocks $\mathcal{B}'_3$ of size 3 in a PBD are given,
as these uniquely determine the blocks
of size 2. The column
$\mathrm{N}'$ gives the number of $w$-subsets of points in an \STS{21} that induce such a
PBD. For $w=6$, we have a Pasch configuration, also known as a
quadrilateral; the number of occurrences is not constant in this case.

For 1000 random graphs $G$ with given degree sequences, the counts
\NDG{} and \NFGW{}
are determined, and the averages are
tabulated in the columns $\overline{\NDG{}}$ and
$\overline{\NFGW{}}$, respectively.
The average computing times, in milliseconds, are
shown in the columns $\overline{\tDG{}}$ and
$\overline{\tFGW{}}$. It turns out that
$\overline{\NDG{}}\cdot\overline{\NFGW{}}$
differs from the average of the product,
$\overline{\NDG{}\cdot \NFGW{}}$,
typically by only a few percent.
The time estimations are incomparable between Table~\ref{tbl:est} and the final computation in Section~\ref{sect:results}, because less optimized versions of the counting algorithms were applied.
Nevertheless, the time estimations are consistent among the entries in Table~\ref{tbl:est} and allow comparisons.

The value of $\mathcal{W}$ is uniquely determined by $G$ for all
instances of Table~\ref{tbl:est} but the degree sequence $3^4 5^{12}$.
Denoting the set of vertices of degree 3 in such a graph by
$\{a,b,c,d\}$, there are three possibilities for $\mathcal{W}$:
\[
\begin{array}{l}
\{\{\},\{a,b\},\{c,d\},\{a,c\},\{b,d\}\},\\
\{\{\},\{a,b\},\{c,d\},\{a,d\},\{b,c\}\},\mbox{\ and}\\
\{\{\},\{a,c\},\{b,d\},\{a,d\},\{b,c\}\}.
\end{array}
\]
The entries in Table~\ref{tbl:est}
for this case are summed over all three subcases.

The choice of $w$ can now be made by comparing the values
of
\begin{align*}
\frac{\overline{\NDG{}} \cdot \overline{\NFGW{}}}
{\overline{\tDG{}} + \overline{\tFGW{}}}
.
\end{align*}
The obvious choice is $w=5$ with $\mathcal{B}'_3 = \{012,034\}$.

\begin{table}[htbp]
\caption{Impact of different choices of $(W,\mathcal{B}')$}\label{tbl:est}
\vspace{-2mm}
\begin{center}
\setlength{\tabcolsep}{3pt}
\begin{tabular}{clrlrrrr}
\hline
\\[-4.3mm]
$w$ & $\mathcal{B}'_3$ & $\mathrm{N}'$ &  $G$ & $\overline{\NDG{}}$ & $\overline{\tDG{}}$
& $\overline{\NFGW{}}$ & $\overline{\tFGW{}}$ \\
\hline
4 & $\{012\}$    &1260 & $2^3 4^{14}$ &
160837000.0 & 362091.0 & 221.7 & 4.7 \\
5 & $\{012,034\}$& 945 & $1^2 5^{14}$ &
12666.6 & 33.9 & 3994.0 & 11.2 \\
  &              &     & $1^1 3^2 5^{13}$ &
10363.2 & 30.1 & 7088.9 & 20.8 \\
  &              &     & $3^4 5^{12}$ &
8714.6 & 26.5 & 26439.6 & 98.5 \\
6 & $\{012,034,$ &                  & $0^1 6^{14}$ &
2.3 & 0.2 & 539449.0 & 858.1 \\
  & \hspace*{2mm}$135,245\}$&       & $2^1 4^1 6^{13}$ &
1.7 & 0.2 & 538846.0 & 817.4 \\
  &                         &       & $4^3 6^{12}$ &
1.3 & 0.2 & 1306050.0 & 2441.3 \\
\hline
\end{tabular}
\end{center}
\end{table}

\subsection{Classifying Graphs With Given Degree Sequences}\label{sect:degree}

Farad\v{z}ev~\cite{F} studied the problem of classifying graphs with given degree
sequences already in the 1970s. The main focus in published studies on
algorithms for classifying graphs with given degree sequences has been on
regular graphs, with work by Meringer~\cite{M} showing the
true potential of such algorithms. Specific algorithms have been
published in particular for graphs with degree 3, that is,
cubic graphs~\cite{B,BGM}.

The graph isomorphism program \texttt{nauty}~\cite{MP} contains a suite
of programs called \texttt{gtools}, which in turn contains the \texttt{geng}
program for generating graphs up to isomorphism and getting
the automorphism groups simultaneously. The \texttt{geng} program is able
to construct graphs with degrees in a given interval. Moreover,
\texttt{geng} can be called from another program, thereby eliminating
a need of extensive computer memory. This is precisely what is required for
the final computation of the counting approach presented here.

\paragraph{STS(21)}

Let us consider the degree sequences for \STSS{21} with $w=5$ and
$\mathcal{B}'_3 = \{012,034\}$ one by one. The final
graphs will in all cases have 16 vertices and 36 edges. As the \texttt{geng}
program will produce also graphs that do not have the desired degree
sequences, the produced graphs must be filtered. One feature of
\texttt{geng} is that it can produce parts of all graphs with little
overhead. In the current work where the computation is
split in parts and distributed, it is indeed important that one does not have
to run the whole graph generation for each part.

Features of the \texttt{geng} program like the aforementioned
one make it an ideal tool here. As the time consumption for graph generation
is negligible with respect to the total computing time of the counting algorithm,
it is not necessary to consider the possibility of additional pruning
in the search tree of \texttt{geng} or using software dedicated to
generating graphs with given degree sequences (but lacking certain features
of \texttt{geng}).

$1^2 5^{14}$: For such a graph $G$, the graph $G'$ induced by the vertices
of degree 5 has degree sequence $5^{14}$, $3^1 5^{13}$, or $4^2 5^{12}$.
Moreover, $|\mbox{Aut}(G)| = 2|\mbox{Aut}(G')|$ in the first
case---the transposition of the two vertices of degree 1 in $G$ is then always in
$\mbox{Aut}(G)$---and $|\mbox{Aut}(G)| = |\mbox{Aut}(G')|$ in the other two cases.
We shall later see that the second of the three subcases can be ignored.

$1^1 3^2 5^{13}$: For such a graph $G$, the graph $G'$ induced by the vertices
of degree at least 3 has degree sequence $2^1 3^1 5^{13}$ or $3^2 4^1 5^{12}$.
In both cases, $|\mbox{Aut}(G)| = |\mbox{Aut}(G')|$.

$3^4 5^{12}$: This case is handled directly.

For the first two degree sequences, there is in all cases a unique way of
adding edges to get $G$ from $G'$.

The parameters of the instances to be computed with \texttt{geng} are summarized
in Table~\ref{tbl:geng}, with the case that can be ignored within parentheses.
In the first two columns, the final and
intermediate degree sequences are given. The columns contain the order of the graph $n$,
the number of edges $e$, the minimum degree $d$, and the maximum degree $D$.

\begin{table}[htbp]
\caption{Parameters of \texttt{geng} instances}\label{tbl:geng}
\begin{center}
\begin{tabular}{llrrrl}
\hline
$S_1$      & $S_2$      & $n$ & $e$ & $d$ & \hspace*{-1mm}$D$\\\hline
$1^2 5^{14}$ & $5^{14}$    & 14  & 35  & 5   & 5\\
\hspace*{-1.6mm}($1^2 5^{14}$ & $3^1 5^{13}$ & 14  & 34  & 3   & 5)\\
$1^2 5^{14}$ & $4^2 5^{12}$ & 14  & 34  & 4   & 5\\
$1^1 3^2 5^{13}$ & $2^1 3^1 5^{13}$ & 15 & 35 & 2 & 5\\
$1^1 3^2 5^{13}$ & $3^2 4^1 5^{12}$ & 15 & 35 & 3 & 5\\
$3^4 5^{12}$ & $3^4 5^{12}$ & 16 & 36 & 3 & 5\\
\hline
\end{tabular}
\end{center}
\end{table}

\subsection{Counting Graph Decompositions}\label{sect:decomp}

Given $G$ and $\mathcal{W}$,
we want to determine \NDG{} and
\NFGW{}. This is here done with algorithms
that explicitly find all such decompositions.

Steiner triple systems are decompositions of complete graphs into
3-cycles (triangles), and for \NDG{} we have decompositions
of other graphs. The problem of finding such decompositions is
recurrent in computational studies of Steiner triple systems~\cite{KO04} and can be
considered within the framework of exact cover. In the \emph{exact
cover problem}, we have a set $S$ and a collection $\mathcal{S}$ of subsets
of $S$, and the decision problem asks whether there exists a partition of $S$
using sets from $\mathcal{S}$. One can further ask for one or all such partitions.
Any graph decomposition problem can be phrased as an instance of the exact
cover problem, where $S$ is the set of edges and $\mathcal{S}$ contains
the set of candidate subgraphs in a decomposition. Instances of the exact cover
problem encountered in
the current work were solved using the \texttt{libexact} software~\cite{KP}.

For the computation of \NFGW{},
experiments show that a two-stage approach analogous to that used in~\cite{HO} is
efficient.
Let $G = (V \setminus W,E)$ and $\mathcal{W}$ be fixed, with
$W_0,W_1,\ldots,W_{w-1}$ in $\mathcal{W}$.

In the first stage,
for $0 \leq i \leq w-1$,
all perfect matchings in the graph induced by $V \setminus (W \cup W_i)$ are determined
(for example, using exact cover) and saved in $\mathcal{F}'_i$.
In the second stage, we want to decompose $G$ by taking exactly one matching from $\mathcal{F}'_i$ for each $i$, $0 \leq i \leq w-1$.
The following theorem shows that it suffices to consider disjointness of matchings.

\begin{theorem}
A set of disjoint matchings of a graph $G$ with one matching from each
of the collections $\mathcal{F}'_i$, $0 \leq i \leq w-1$, decomposes $G$.
\end{theorem}

\begin{proof}
Each edge of the matchings is an edge of $G$.
As the matchings are disjoint, each edge of $G$ occurs in at
most one matching. The theorem now follows as all matchings
in $\mathcal{F}'_i$ for a given $i$ have the same number of edges.
\end{proof}

It is straightforward to implement a backtrack search for this problem,
cf.~\cite[p.~36]{OS}. With a very small value of $w$ one can
use nested loops instead. If matchings are stored as bitmaps, then
disjointness of matchings can be determined with an AND operation.
Also \texttt{libexact} can solve instances of this problem.

We shall now see that the number of solutions can actually be
counted without traversing the search tree to its leaves.

\begin{theorem}\label{thm:last}
A set of disjoint matchings, one from each of $w-1$ collections $\mathcal{F}'_i$,
can be extended in a unique way to a decomposition of $G = (V \setminus W,E)$ with
disjoint matchings from all collections $\mathcal{F}'_i$.
\end{theorem}

\begin{proof}
For a given $i$, $0 \leq i \leq w-1$, every matching of $\mathcal{F}'_i$
saturates the same subset of $V \setminus W$. After deleting the edges
of a set of $w-1$ disjoint matchings, one from each of $w-1$ collections $\mathcal{F}'_i$, from $E$, we get a graph
with vertex degrees 0 and 1, that is, a matching, which necessarily
occurs in the collection $\mathcal{F}'_i$ from which no matching has yet been taken.
\end{proof}

\begin{theorem}\label{thm:secondlast}
Consider a collection $\mathcal{S}$ of $w-2$ disjoint matchings from different
collections $\mathcal{F}'_i$,
and assume that no matching is included from the collections $\mathcal{F}'_a$
and $\mathcal{F}'_b$, $0 \leq a < b \leq w-1$. Further let $M_a$ ($M_b$)
be the number of matchings in $\mathcal{F}'_a$ ($\mathcal{F}'_b$) that are disjoint
from all matchings of $\mathcal{S}$. Then $M_a = M_b$ and this is
the number of ways $\mathcal{S}$ can be completed to a decomposition of $G$ with
elements from all collections $\mathcal{F}'_i$.
\end{theorem}

\begin{proof}
By Theorem~\ref{thm:last}, after extending $\mathcal{S}$ with a matching from
$\mathcal{F}'_a$, there is exactly one possibility of completing the decomposition
in the desired way, so the total number of completions is $M_a$. In an analogous
way, by first extending with a matching from $\mathcal{F}'_b$, we get that the
total number is $M_b$. It then follows that $M_a = M_b$.
\end{proof}

When some of the sets $W_i$ coincide,
we can combine the corresponding,
identical collections $\mathcal{F}'_i$ to get collections $\mathcal{F}''_j$ for
some $0 \leq j \leq w'$ where $w'<w$. The number of matchings to be
taken from a collection $\mathcal{F}''_j$ is precisely the number of identical collections
combined, and the matchings can now be taken with respect to some defined
total \emph{order}. In this framework, we say that a matching is \emph{compatible} with
a collection $\mathcal{S}$ of matchings if it is disjoint from the matchings
in $\mathcal{S}$ and greater than the matchings in $\mathcal{S}$ that
come from the same collection $\mathcal{F}''_j$.

In the framework of $\mathcal{F}''_j$ we do not
get quite the same situation as in Theorem~\ref{thm:secondlast}. Namely,
when picking the second to last matching it may be that the unique final matching
is not a compatible candidate. However, it certainly is compatible if the total number of
matchings to be taken from its collection $\mathcal{F}''_j$ is 1. In the situation
when we have a collection from which exactly 2 matchings should be taken, we can
avoid considering it during the search, but still maintain its subset of compatible
candidates; if it has $M$ compatible candidates when two matchings are missing,
then the number of completions is $M/2$.

\paragraph{STS(21)}

When counting \STSS{21} with $w=5$ and $\mathcal{B}'_3 = \{012,034\}$,
there are 36 edges in a graph $G$, so bitmaps of matchings fit in
64-bit words. For the degree sequence $3^4 5^{12}$ all $W_i$ are distinct.
By Theorem~\ref{thm:secondlast},
we then need only carry out the search to
the level with $w-2 = 3$ matchings.
The possibilities of $\mathcal{W}$ listed at
the end of Section~\ref{sect:part} show that in each case we need one
matching with 8 edges and four matchings with 7 edges. The number of candidates
is larger for 8 edges than for 7 edges, so we have used Theorem~\ref{thm:secondlast}
and chosen to ignore the matchings with 8 edges in the counting. Obviously,
we then need not even determine the matchings with 8 edges, which speeds up the
algorithm even further.

In the case of \STSS{21} with $w=5$ and $\mathcal{B}'_3 = \{012,034\}$, we get
situations where not all $W_i$ are distinct for
the degree sequences $1^2 5^{14}$ and $1^1 3^2 5^{13}$. For these two
degree sequences, the five values of $W_i$ are, respectively,

\[
\begin{array}{l}
\{\},\{a,b\},\{a,b\},\{a,b\},\{a,b\}\mbox{\ and}\\
\{\},\{a,b\},\{a,b\},\{a,c\},\{a,c\}.
\end{array}
\]

Note that in the former case, with degree sequence $1^2 5^{14}$, the vertices
$a$ and $b$ cannot have the same (unique) neighbor, as one of the matchings
(corresponding to $\{\}$ in $\mathcal{W}$) contains edges saturating all vertices.
A search algorithm will immediately realize this, but the main importance of this
observation is that one of the cases in Table~\ref{tbl:geng} can be ignored in the
classification of graphs $G$.

For the degree sequences $1^2 5^{14}$ and $1^1 3^2 5^{13}$,
in the first stage of the two-stage approach for computing
\NFGW{} we compute 2 and 3
collections of matchings, respectively. For the
degree sequence $1^2 5^{14}$, we search for a decomposition
with 1 matching from one
collection and 4 matchings from the other, and for the degree
sequence $1^1 3^2 5^{13}$, we search for a decomposition with
1 matching from one collection and 2 matchings from each of the two other
collections. These cases are not time-critical for the complete algorithm, so
\texttt{libexact} was used.

\subsection{Counting Labeled Systems}\label{sect:label}

The main question still remains: how to get the total number of labeled
Steiner triple systems from the data obtained as described in Section~\ref{sect:decomp}?
Let $\mathcal{G}$ denote the set of all pairs $(G,\mathcal{W})$ to consider for
given values of $v$, $w$, and $\mathcal{B}'_3$.

\begin{theorem}\label{thm:count}
The total number of labeled\/ \STSS{v} is
\begin{equation}\label{eq:total}
\frac{1}{\mathrm{N}'}
\cdot
\sum_{(G,\mathcal{W}) \in \mathcal{G}}
\frac{Kv!\NDG{}\NFGW{}}{w!|\mathrm{Aut}(G)|},
\end{equation}
where $K$ is the number of ways to complete the triples and matchings
counted in\/ \NDG{} and\/
\NFGW{} to\/ \STSS{v}.
\end{theorem}

\begin{proof}
We count the number of pairs $(A,B)$, where $A$ is a Steiner triple system and
$B$ is a set of blocks of $A$ isomorphic to $\mathcal{B}'_3$. By definition,
$\mbox{N}'$ is the number of sets of blocks of $A$ isomorphic to
$\mathcal{B}'_3$, so we have to divide the final count by $\mbox{N}'$ to get
the desired result.

The number of ways of choosing the $v-w$ points of $G$ out of the $v$ points is
$\binom{v}{v-w}$. The number of labeled graphs on these points is further
$(v-w)!/|\mbox{Aut}(G)|$. For each such graph we have
\NDG{} solutions for the $\mathcal{D}$ part and
\NFGW{} solutions for the matchings of
the $\mathcal{F}$ part.

As $K$ is the number of ways such a structure can be completed to
\STSS{v}, the theorem follows.
\end{proof}

We still need to elaborate on $K$, the number of ways to complete
the parts. First, we want to find the number of PBDs of the given
type on the point set $W$ with a completion of the blocks of size 2
using points from the (fixed) set $W' = \bigcup_{p \in W}W_p$. This
number can be obtained combinatorially or computationally.

Finally, when forming $\mathcal{F}$ from
a collection of matchings, there are several possibilities in case of
coinciding sets $W_i$. This is seen also in the forming of the collections
$\mathcal{F}''_j$, $0 \leq j \leq w'-1$, from the collections
$\mathcal{F}'_i$, $0 \leq i \leq w-1$. Specifically, letting $P_j$ be
the number of collections $\mathcal{F}'_i$ combined into
$\mathcal{F}''_j$, the desired count is $\prod_{j=0}^{w'-1} P_j!$.
Note, however, that for the smallest instances a collection
$\mathcal{F}''_j$ may contain an empty set; then $P_j!$ should
be replaced by 1.

\paragraph{STS(21)}

The following examples also apply to STSs of other orders than 21.
With $w=5$ and $\mathcal{B}'_3 = \{012,034\}$, we have the three possible
situations depicted in Figure~\ref{fig:five}.
The desired numbers are obtained by dividing 5! with the order of the
automorphism groups of these structures, which are 4, 2, and 1,
respectively. We then get the following values of $K$ for the
three degree sequences:

\[
\begin{array}{l}
1^2 5^{14}: K = (5!/4)\cdot 4! = 720,\\
1^1 3^2 5^{13}: K = (5!/2)\cdot 2! \cdot 2! = 240,\\
3^4 5^{12}: K = (5!/1)\cdot 1 = 120.\\
\end{array}
\]

\noindent
For orders other than 21, we get an exception due to an empty set
in $\mathcal{F}''_j$ in exactly one situation. Namely, for order 7
and the degree sequence $1^2 5^{14}$, we get $K = (5!/4)\cdot 1 = 30$.

\begin{figure}[htbp]
\hfill
\begin{subfigure}[b]{0.3\textwidth}
\begin{tikzpicture}[scale=0.9]
\coordinate (X0) at (0:0);
\coordinate (X2) at (255:4);
\coordinate (X1) at ($(X0)!0.5!(X2)$);
\coordinate (X4) at (285:4);
\coordinate (X3) at ($(X0)!0.5!(X4)$);

\draw[line width=2] (X0)--(X1)--(X2);
\draw[line width=2] (X0)--(X3)--(X4);

\coordinate (Y0) at ([xshift=1.0cm]$(X3)!0.5!(X4)$);
\coordinate (Y1) at ([xshift=2.0cm]$(X3)!0.5!(X4)$);

\draw plot [smooth] coordinates { (X1) (X3) (Y0) };
\draw plot [smooth] coordinates { (X2) (X4) (Y0) };
\draw[dashed] plot [smooth] coordinates { (X1) (X4) (Y1) };
\draw[dashed] plot [smooth] coordinates { (X2) (X3) (Y1) };

\node[draw,fill=white,circle,inner sep=3] at (X0) {0};
\node[draw,fill=white,circle,inner sep=3] at (X1) {1};
\node[draw,fill=white,circle,inner sep=3] at (X2) {2};
\node[draw,fill=white,circle,inner sep=3] at (X3) {3};
\node[draw,fill=white,circle,inner sep=3] at (X4) {4};

\node[draw,fill=white,circle,inner sep=3] at (Y0) {5};
\node[draw,fill=white,circle,inner sep=3] at (Y1) {6};
\end{tikzpicture}
\caption{$1^2 5^{14}$}
\end{subfigure}
\hfill
\begin{subfigure}[b]{0.3\textwidth}
\begin{tikzpicture}[scale=0.9]
\coordinate (X0) at (0:0);
\coordinate (X2) at (255:4);
\coordinate (X1) at ($(X0)!0.5!(X2)$);
\coordinate (X4) at (285:4);
\coordinate (X3) at ($(X0)!0.5!(X4)$);

\draw[line width=2] (X0)--(X1)--(X2);
\draw[line width=2] (X0)--(X3)--(X4);

\coordinate (Y0) at ([shift=({1.5cm,1cm})]$(X3)!0.5!(X4)$);
\coordinate (Y1) at ([shift=({1.5cm,0cm})]$(X3)!0.5!(X4)$);
\coordinate (Y2) at ([shift=({1.5cm,-1cm})]$(X3)!0.5!(X4)$);

\draw plot [smooth] coordinates { (X1) (X3) (Y1) };
\draw plot [smooth] coordinates { (X2) (X4) (Y1) };
\draw[dashed] plot [smooth] coordinates { (X1) (X4) (Y2) };
\draw[dashed] plot [smooth] coordinates { (X2) (X3) (Y0) };

\node[draw,fill=white,circle,inner sep=3] at (X0) {0};
\node[draw,fill=white,circle,inner sep=3] at (X1) {1};
\node[draw,fill=white,circle,inner sep=3] at (X2) {2};
\node[draw,fill=white,circle,inner sep=3] at (X3) {3};
\node[draw,fill=white,circle,inner sep=3] at (X4) {4};

\node[draw,fill=white,circle,inner sep=3] at (Y0) {5};
\node[draw,fill=white,circle,inner sep=3] at (Y1) {6};
\node[draw,fill=white,circle,inner sep=3] at (Y2) {7};
\end{tikzpicture}
\caption{$1^1 3^2 5^{13}$}
\end{subfigure}
\hfill
\begin{subfigure}[b]{0.3\textwidth}
\begin{tikzpicture}[scale=0.9]
\coordinate (X0) at (0:0);
\coordinate (X2) at (255:4);
\coordinate (X1) at ($(X0)!0.5!(X2)$);
\coordinate (X4) at (285:4);
\coordinate (X3) at ($(X0)!0.5!(X4)$);

\draw[line width=2] (X0)--(X1)--(X2);
\draw[line width=2] (X0)--(X3)--(X4);

\coordinate (Y0) at ([shift=({1.5cm,2.0cm})]$(X3)!0.5!(X4)$);
\coordinate (Y1) at ([shift=({1.5cm,1.0cm})]$(X3)!0.5!(X4)$);
\coordinate (Y2) at ([shift=({1.5cm,-0.0cm})]$(X3)!0.5!(X4)$);
\coordinate (Y3) at ([shift=({1.5cm,-1.0cm})]$(X3)!0.5!(X4)$);

\draw plot [smooth] coordinates { (X1) (X3) (Y1) };
\draw plot [smooth] coordinates { (X2) (X4) (Y2) };
\draw[dashed] plot [smooth] coordinates { (X1) (X4) (Y3) };
\draw[dashed] plot [smooth] coordinates { (X2) (X3) (Y0) };

\node[draw,fill=white,circle,inner sep=3] at (X0) {0};
\node[draw,fill=white,circle,inner sep=3] at (X1) {1};
\node[draw,fill=white,circle,inner sep=3] at (X2) {2};
\node[draw,fill=white,circle,inner sep=3] at (X3) {3};
\node[draw,fill=white,circle,inner sep=3] at (X4) {4};

\node[draw,fill=white,circle,inner sep=3] at (Y0) {5};
\node[draw,fill=white,circle,inner sep=3] at (Y1) {6};
\node[draw,fill=white,circle,inner sep=3] at (Y2) {7};
\node[draw,fill=white,circle,inner sep=3] at (Y3) {8};
\end{tikzpicture}
\caption{$3^4 5^{12}$}
\end{subfigure}
\hfill
\caption{Connecting PBDs to additional points}\label{fig:five}
\end{figure}
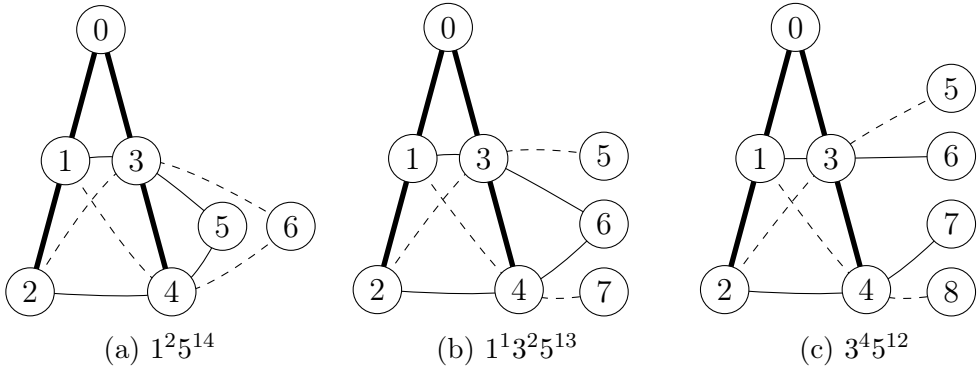

\subsection{Counting Isomorphism Classes}\label{sect:iso}

Denote by $N_{v,i}$ the number of isomorphism classes of \STSS{v}
with an automorphism group of order $i$. By the Orbit--Stabilizer
theorem, the number of labeled Steiner triple systems is
\begin{equation}\label{eq:sum}
v!\sum_i \frac{N_{v,i}}{i}.
\end{equation}
If, for some fixed $v$, we know the number of labeled systems
and $N_{v,i}$ for all $i \geq 2$,
we arrive at an equation where only $N_{v,1}$ is unknown and
can be determined. An error has occurred if the result is not
an integer. The total number of isomorphism classes
is obviously $\sum_i N_{v,i}$.

The sum~\eqref{eq:sum} contains the term $v!N_{v,1}$. Assuming
correctness of the classification of the \STSS{v} with nontrivial
automorphisms, this means that
an incorrect number of labeled Steiner triple systems will be accepted
only if it differs from the correct number by a multiple of
$v!$.

Assume that the value of
$\NDG{}\NFGW{}$
in~\eqref{eq:total} is incorrect, and denote the difference between the correct
and the incorrect value by $E$. For such an error to go undetected
$Kv!E/(w!|\mathrm{Aut}(G)|\mathrm{N}')$ should be divisible by $v!$,
that is, $KE/(w!|\mathrm{Aut}(G)|\mathrm{N}')$ should be an
integer. In the common situation of $K = w!$ and $|\mathrm{Aut}(G)| = 1$,
we get that $E$ should be divisible by $\mathrm{N}'$. Also in other
cases---including missing graphs, run-time errors, and multiple counting
errors---divisibility of some integer by $\mathrm{N}'$
is a core question and gives an impression of the probability of
errors going undetected.

\paragraph{STS(21)}

A classification of \STSS{21} with nontrivial automorphisms
was carried out by Kaski~\cite{K} and validated by double
counting. The distribution of such systems based on the order
of the automorphism groups is displayed in
Table~\ref{tbl:nontrivial}.

\begin{table}[htbp]
\caption{The \STS{21} with nontrivial automorphisms}\label{tbl:nontrivial}
\begin{center}
\begin{tabular}{rr@{\hskip 30pt}rr@{\hskip 30pt}rr}
\hline
$|\mathrm{Aut}(X)|$ & N & $|\mathrm{Aut}(X)|$ & N & $|\mathrm{Aut}(X)|$ & N \\
\hline
2    &    60{,}588{,}267  & 14   &    14 & 54   &    1  \\
3    &    1{,}732{,}131  & 16   &    12 & 72   &    5  \\
4    &    11{,}467  & 18   &    33 & 108  &    1  \\
5    &    1{,}772  & 21   &    10 & 126  &    2  \\
6    &    2{,}379  & 24   &    19 & 144  &    1  \\
7    &    66  & 27   &    3  & 294  &    1  \\
8    &    222  & 36   &    5  & 504  &    1  \\
9    &    109  & 42   &    7  & 882  &    1  \\
12   &    85  & 48   &    2  & 1008 &    1  \\
\hline
\end{tabular}
\end{center}
\end{table}

As argued above,
the results produced by the software used in the computations,
including \texttt{geng} and our own programs, are subject to the
overall validation. For the divisibility argument discussed, we
can be somewhat more specific in the case of the current
parameters: \STSS{21}, $w=5$, and $\mathcal{B}'_3 = \{012,034\}$.

For the degree sequences $1^2 5^{14}$, $1^1 3^2 5^{13}$, and
$3^4 5^{12}$ and respective errors $E_1$, $E_2$, and $E_3$,
the impact of these errors on the total sum is
\[
\frac{6\cdot 21!E_1}{945|\mbox{Aut}(G)|},\
\frac{2\cdot 21!E_2}{945|\mbox{Aut}(G)|}, \mbox{\ and \ }
\frac{21!E_3}{945|\mbox{Aut}(G)|}.
\]
As $|\mbox{Aut}(G)|=1$ for the vast majority of graphs, we
are essentially asking whether $6E_1+2E_2+E_3$ is divisible
by 945. With a uniform distribution of that sum over all possible
values mod 945, this probability would obviously be $1/945$.

The graphs with degree sequence $5^{14}$ are regular graphs and
their number can be found in the literature~\cite{M,R}.

\section{The Main Result}\label{sect:results}

The main computations were carried out in a cluster, and the computing times
announced here are for computations on the equivalent of one core of a
2.4-GHz Intel Xeon E5-2665. The total computing time was approximately 82
core-years. The classification of the \STSS{21} with nontrivial automorphisms
in~\cite{K} took 25 core-hours with a 2-GHz CPU.

Due to the very large number of graphs with the given degree sequences, these
were not saved, but they can at any point be reproduced with \texttt{geng}.
The number of isomorphism classes of graphs with degree sequences in column~$S_2$
in Table~\ref{tbl:geng}
are shown in Table~\ref{tbl:nonregular}
and grouped based on the order of the automorphism group and the degree sequence.
Recall that, as described in Section~\ref{sect:degree}, any graph with degree sequence $S_2$ can be extended
in a unique way to a graph with the corresponding degree sequence
$S_1$ in Table~\ref{tbl:geng} (by adding vertices of degree one).
Moreover, the order of the automorphism group remains unchanged in all but one
of the extensions: it doubles when going from $5^{14}$ to $1^2 5^{14}$.

The last row in Table~\ref{tbl:nonregular} shows the amount of core-hours
for all computations in the respective cases, including the
generation of graphs and the computation of \NDG{} and
\NFGW{}.

\begin{table}[htbp]
\caption{Graphs with certain degree sequences}\label{tbl:nonregular}
\begin{center}
\footnotesize
\begin{tabular}{rrrrrr}
\hline
|\mbox{Aut}(G)| & $5^{14}$ & $4^2 5^{12}$ & $2^1 3^1 5^{13}$ & $3^2 4^1 5^{12}$ & $3^4 5^{12}$\\\hline
1 & 3{,}063{,}687 & 143{,}619{,}210 & 714{,}665{,}364 & 9{,}352{,}704{,}447 & 50{,}268{,}563{,}872 \\
2 & 344{,}187 & 11{,}440{,}917 & 52{,}316{,}559 & 565{,}093{,}004 & 3{,}265{,}605{,}905 \\
3 &   & 47 & 24 & 50 & 1{,}507 \\
4 & 42{,}047 & 959{,}414 & 3{,}865{,}385 & 33{,}876{,}404 & 185{,}390{,}630 \\
6 & 461 & 19{,}402 & 103{,}883 & 952{,}540 & 5{,}479{,}846 \\
7 & 1 &   &   &   &   \\
8 & 6{,}713 & 89{,}361 & 278{,}358 & 2{,}045{,}349 & 11{,}122{,}205 \\
12 & 456 & 9{,}664 & 42{,}587 & 290{,}483 & 1{,}373{,}479 \\
14 & 5 &   &   &   &   \\
16 & 1{,}103 & 8{,}785 & 19{,}688 & 130{,}440 & 722{,}715 \\
18 &   &   &   &   & 53 \\
24 & 251 & 3{,}088 & 9{,}996 & 56{,}321 & 265{,}512 \\
28 & 4 &   &   &   &   \\
32 & 209 & 993 & 1{,}411 & 9{,}181 & 53{,}049 \\
36 & 4 & 89 & 359 & 1{,}785 & 5{,}509 \\
48 & 114 & 767 & 1{,}788 & 9{,}455 & 46{,}312 \\
60 & 2 &   &   &   & 1 \\
64 & 37 & 123 & 113 & 740 & 4{,}504 \\
72 & 9 & 98 & 297 & 1{,}082 & 3{,}929 \\
80 & 3 &   &   &   &   \\
96 & 30 & 161 & 318 & 1{,}590 & 7{,}661 \\
108 &   &   &   &   & 1 \\
120 &   &   &   & 1 & 3 \\
128 & 14 & 24 & 6 & 75 & 526 \\
144 & 11 & 60 & 101 & 359 & 1{,}576 \\
160 &   &   &   &   & 1 \\
192 & 10 & 35 & 69 & 281 & 1{,}420 \\
216 &   & 3 & 3 & 6 & 41 \\
240 &   &   &   & 1 & 5 \\
256 & 5 & 8 &   & 5 & 74 \\
288 & 7 & 20 & 35 & 116 & 489 \\
320 &   &   &   &   & 1 \\
384 & 2 & 8 & 7 & 40 & 248 \\
432 & 1 & 4 & 2 & 7 & 54 \\
480 &   &   &   & 1 & 7 \\
512 & 1 & 1 &   & 1 & 17 \\
576 & 2 & 11 & 10 & 35 & 160 \\
720 &   &   & 8 & 94 & 434 \\
768 & 2 & 3 & 2 & 4 & 37 \\
864 &   & 2 & 1 &   & 26 \\
960 &   &   &   &   & 1 \\
1024 &   &   &   & 1 & 4 \\
\hline
&&&&&(cont.)
\end{tabular}
\end{center}
\end{table}

\addtocounter{table}{-1}

\begin{table}[htbp]
\caption{Graphs with certain degree sequences (cont.)}\label{tbl:nonregular2}
\begin{center}
\footnotesize
\begin{tabular}{rrrrrr}
\hline
|\mbox{Aut}(G)| & $5^{14}$ & $4^2 5^{12}$ & $2^1 3^1 5^{13}$ & $3^2 4^1 5^{12}$ & $3^4 5^{12}$\\\hline
1152 & 1 & 1 & 1 & 8 & 53 \\
1296 &   &   &   &   & 1 \\
1440 &   & 3 & 13 & 86 & 283 \\
1536 &   &   & 1 &   & 11 \\
1728 & 1 & 1 &   & 1 & 17 \\
1792 & 1 &   &   &   &   \\
1920 &   &   &   &   & 1 \\
2048 &   &   &   &   & 1 \\
2160 &   &   &   &   & 2 \\
2304 & 1 &   & 1 & 3 & 13 \\
2592 &   &   &   &   & 1 \\
2880 & 1 & 5 & 12 & 40 & 133 \\
3072 &   &   &   &   & 3 \\
3456 &   &   &   &   & 3 \\
3840 &   &   &   &   & 1 \\
4096 &   &   &   &   & 1 \\
4320 &   &   &   & 1 & 6 \\
4608 &   &   &   &   & 5 \\
5120 &   &   &   &   & 1 \\
5760 &   & 2 & 2 & 9 & 36 \\
6144 &   & 1 &   &   &   \\
6912 &   &   &   &   & 4 \\
8640 &   & 2 & 3 & 5 & 6 \\
9216 &   &   &   &   & 2 \\
10368 &   &   &   &   & 4 \\
11520 & 1 &   & 1 & 3 & 10 \\
13824 &   &   &   &   & 1 \\
17280 &   &   &   & 1 & 4 \\
18432 &   &   &   &   & 1 \\
23040 &   & 1 &   &   & 4 \\
25920 &   &   &   &   & 3 \\
34560 &   &   &   &   & 4 \\
43200 & 1 &   &   &   &   \\
46080 &   &   &   &   & 1 \\
51840 &   & 1 &   &   & 1 \\
55296 &   &   &   &   & 1 \\
69120 &   &   &   &   & 1 \\
92160 & 1 &   &   &   &   \\
207360 &   &   &   &   & 1 \\
460800 &   &   &   &   & 1 \\
24883200 &   &   &   &   & 1 \\
\hline
Total & 3{,}459{,}386 & 156{,}152{,}315 & 771{,}306{,}408 & 9{,}955{,}174{,}055 & 53{,}738{,}652{,}436 \\
Core-h & 80 & 1{,}799 & 12{,}745 & 135{,}073 & 572{,}276 \\
\hline
\end{tabular}
\end{center}
\end{table}

In Table~\ref{tbl:countspergraph}, we split the value of~\eqref{eq:total}
into partial sums, one for each degree
sequence. (Notice that we in general cannot assume that such
partial sums are integers.)

\begin{table}[htbp]
\caption{Partial sums of~\eqref{eq:total} per degree sequence}\label{tbl:countspergraph}
\begin{center}
\begin{tabular}{llr}
\hline
$S_1$ & $S_2$ &  Partial sum\\
\hline
$1^2 5^{14}$ & $5^{14}$ &  133{,}088{,}588{,}244{,}979{,}214{,}201{,}168{,}855{,}040{,}000 \\
$1^2 5^{14}$ & $4^2 5^{12}$&2{,}538{,}696{,}865{,}871{,}668{,}928{,}235{,}196{,}907{,}520{,}000 \\
$1^1 3^2 5^{13}$& $2^1 3^1 5^{13}$ &10{,}154{,}787{,}463{,}486{,}675{,}712{,}940{,}787{,}630{,}080{,}000 \\
$1^1 3^2 5^{13}$& $3^2 4^1 5^{12}$ &77{,}792{,}298{,}507{,}007{,}219{,}219{,}438{,}529{,}150{,}976{,}000 \\
$3^4 5^{12}$   & $3^4 5^{12}$ &    665{,}333{,}309{,}624{,}296{,}811{,}889{,}899{,}926{,}978{,}560{,}000 \\
\hline
Total && 755{,}952{,}181{,}048{,}907{,}354{,}964{,}715{,}609{,}522{,}176{,}000 \\
\hline
\end{tabular}
\end{center}
\end{table}

By Table~\ref{tbl:countspergraph}, the total number of labeled \STS{21} is
\begin{align*}
755{,}952{,}181{,}048{,}907{,}354{,}964{,}715{,}609{,}522{,}176{,}000.
\end{align*}
By~\eqref{eq:sum} and Table~\ref{tbl:nontrivial}, we then get that
the number of isomorphism classes of \STSS{21} with trivial automorphism
group is
\begin{align*}
N_{21,1} = 14{,}796{,}207{,}455{,}537{,}154.
\end{align*}
Finally, once again utilizing Table~\ref{tbl:nontrivial},
this yields the total number of isomorphism classes of \STSS{21},
\begin{align*}
14{,}796{,}207{,}517{,}873{,}771,
\end{align*}
which completes a computer-aided proof of Theorem~\ref{thm:main}.

Using a random hypergraph model, it is conjectured in~\cite{HO} that the proportion of \STSS{v} containing at least one \STS*{7} is $\alpha = 1-e^{-1/168} \approx 0.00593$ for large $v$. By~\cite{HO}, the number of isomorphism classes of \STSS{21} that contain at least one \STS*{7} is $116{,}635{,}963{,}205{,}551$.
The real proportion for \STSS{21} is hence
\begin{align*}
\frac{
116{,}635{,}963{,}205{,}551
}{
14{,}796{,}207{,}517{,}873{,}771
}
\approx
0.00788
\approx
1.3
\alpha,
\end{align*}
which approximately coincides with the real proportion of $1.3\alpha$ in the case of \STSS{19} and is in line with the experimental results
of~\cite{HO2}.

The approach was also applied with the same main parameters
($w=5$ and $\mathcal{B}'_3 = \{012,034\}$) to the cases of
\STSS{13}, \STSS{15}, and \STSS{19}. The
number of labeled \STSS{19} was obtained in about 60 core-hours.
We did not investigate whether $w=5$ is the optimal choice for \STS{19}.

\section*{Acknowledgements}

The authors thank the referees for valuable comments.

\end{document}